\documentclass[11pt,reqno]{amsart}
\usepackage{amsmath}
\usepackage{amssymb,latexsym}
\usepackage[draft=true]{hyperref}
\usepackage{cite}

\theoremstyle{plain}
\newtheorem{theorem}{Theorem}

\newtheorem{corollary}{Corollary}

\theoremstyle{definition}
\newtheorem{definition}{Definition}
\theoremstyle{remark}
\newtheorem{remark}{Remark}

\numberwithin{equation}{section} 
\input{tcilatex}

\begin{document}
\title[On A New Type Multivariable Hypergeometric Functions]{On A New Type
Multivariable Hypergeometric Functions}
\author{D. Korkmaz-Duzgun}
\address{Kafkas University\\
Faculty of Economics and Administrative Sciences\\
Department of Business Administration\\
Central Campus TR-36100 Kars\\
Turkey}
\email{dgkorkmaz@gazi.edu.tr, dryekorkmaz@gmail.com}
\author{ E. Erku\c{s}-Duman}
\address{Gazi University\\
Faculty of Science\\
Department of Mathematics\\
Teknikokullar TR-06500 Ankara\\
Turkey}
\email{eduman@gazi.edu.tr, eerkusduman@gmail.com}

\begin{abstract}
In this paper, we define a new type multivariable hypergeometric function.
Then, we obtain some generating functions for these functions. Furthermore,
we derive various families of multilinear and multilateral generating
functions for these multivariable hypergeometric functions and their special
cases are also given.
\end{abstract}

\maketitle

\section{Introduction}

Multivariable hypergeometric functions have a great importance in special
functions theory. A lot of integral representations and transformations
obtained, generating functions of which used in a lot of scopes has been
studied, for these functions.

A natural generalization of an arbitrary number of $p$\ numerator and $q$\
denominator parameters $(p,q\in N_{0}=\left \{ 0\right \} \cup \mathbb{N})$
is called and denoted by the generalized hypergeometric series $_{p}F_{q} $\
defined by \cite{SM} 
\begin{align}
_{p}F_{q}\left[ 
\begin{array}{c}
\alpha _{1},...,\alpha _{p}; \\ 
\beta _{1},...,\beta _{q};%
\end{array}
\ \ z\right] &= \sum \limits_{n=0}^{\infty }\frac{(\alpha _{1})_{n}\ ...\
(\alpha _{p})_{n}}{(\beta _{1})_{n}\ ...\ (\beta _{q})_{n}}\frac{z^{n}}{n!}
\label{a1} \\
&= _{p}F_{q}\left( \alpha _{1},...,\alpha _{p};\beta
_{1},...,\beta_{q};z\right) .  \notag
\end{align}
where $(\lambda )_{\nu }$ denotes the Pochhammer symbol defined (in terms of
gamma function) by 
\begin{align*}
(\lambda )_{\nu } &= \frac{\Gamma (\lambda +\nu )}{\Gamma (\lambda )}\ \
(\lambda \in \mathbb{C}\setminus \mathbb{Z}_{0}^{-}) \medskip \\
&=\left \{ 
\begin{array}{ll}
1, & \text{ if }\nu =0;\text{ }\lambda \in \mathbb{C}\backslash \{0\} \\ 
\lambda (\lambda +1)...(\lambda +n-1), & \text{ if }\nu =n\in \mathbb{N}; 
\text{ }\lambda \in \mathbb{C}%
\end{array}
\right. ,
\end{align*}
$Z_{0}^{-}$\ denotes the set of nonpositive integers, and $\Gamma (\lambda )$
is gamma function. In (\ref{a1}), the absence of parameters $p$ or $q$ is
emphasized by a dash (interpreting an empty product as 1). For example, if
no numerator or denominator parameters are present, i.e., $p=q=0$, the
result is $_{0}F_{0}\left( -;-;z\right) =\sum\limits_{n=0}^{\infty }\frac{%
z^{n}}{n!}$ similarly if $p=0$ and $q=1$, then $_{0}F_{1}\left( -;\beta
;z\right) =\sum\limits_{n=0}^{\infty }\frac{z^{n}}{(\beta )_{n}\ n!}$ and
finally if $p=1 $ and $q=0$, then $_{1}F_{0}\left( \alpha ;-;z\right)
=\sum\limits_{n=0}^{\infty }\frac{(\alpha )_{n}\ z^{n}}{n!}$.

In (\ref{a1}), if we choose $p=2$, $q=1$, we get hypergeometric function and
if we also set $p=1$, $q=1$, we obtain confluent hypergeometric function $%
\Phi $.

On the other hand, the second kind Appell functions, the first kind
Lauricella functions, the fourth kind Horn functions and the multivariable
fourth kind Horn functions \cite{SM,H} that like respectively, 
\begin{align*}
F_{2}&\left( {\small \alpha ,\beta }_{1}{\small ,\beta }_{2}{\small ;\gamma }%
_{1}{\small ,\gamma }_{2}{\small ;x}_{1}{\small ,x}_{2}\right)
=\sum\limits_{m,r=0}^{\infty }{\small (\alpha )}_{m+r}\frac{(\beta
_{1})_{m}(\beta _{2})_{r}}{(\gamma _{1})_{m}(\gamma _{2})_{r}}\frac{x_{1}^{m}%
}{m!}\frac{x_{2}^{r}}{r!}, \\
&\left( \left \vert x_{1}\right \vert +\left \vert x_{2}\right \vert
<1\right)
\end{align*}
and 
\begin{align*}
F_{A}^{\left( r\right) }&\left( {\small \alpha ,\beta }_{1}{\small %
,...,\beta }_{r}{\small ;\gamma }_{1}{\small ,...,\gamma }_{r}{\small ;x}_{1}%
{\small ,...,x}_{r}\right) =\sum \limits_{m_{1},...,m_{r}=0}^{\infty }%
{\small (\alpha )}_{m_{1}+...+m_{r}}\frac{(\beta _{1})_{m_{1}}...(\beta
_{r})_{m_{r}}}{(\gamma_{1})_{m_{1}}...(\gamma _{r})_{m_{r}}}\frac{%
x_{1}^{m_{1}}}{m_{1}!}{\small ...}\frac{x_{r}^{m_{r}}}{m_{r}!}, \\
&\ \ \ \ \left( \left \vert x_{1}\right \vert +...+\left \vert x_{r}\right
\vert <1\right)
\end{align*}
and 
\begin{align*}
H_{4}&\left( \alpha ,\beta ;\gamma _{1},\gamma _{2};x_{1},x_{2}\right)
=\sum\limits_{m,r=0}^{\infty }\frac{(\alpha )_{2m+r}\ (\beta )_{r}}{%
(\gamma_{1})_{m}(\gamma _{2})_{r}}\frac{x_{1}^{m}}{m!}\frac{x_{2}^{r}}{r!},
\\
&\ \ \ \left( 2\sqrt{\left \vert x_{1}\right \vert }+\left \vert x_{2}\right
\vert <1\right)
\end{align*}
and 
\begin{align*}
^{(k)}H_{4}^{(r)}&\left( \alpha ,\beta _{k+1},...,\beta
_{r};\gamma_{1},...,\gamma _{r};x_{1},...,x_{k},x_{k+1},...,x_{r}\right) \\
&=\sum \limits_{m_{1},...,m_{r}=0}^{\infty }\frac{%
(\alpha)_{2(m_{1}+...+m_{k})+m_{k+1}+...+m_{r}}~(\beta
_{k+1})_{m_{k+1}}...(\beta_{r})_{m_{r}}}{(\gamma _{1})_{m_{1}}...(\gamma
_{r})_{m_{r}}}\frac{x_{1}^{m_{1}}}{m_{1}!}...\frac{x_{r}^{m_{r}}}{m_{r}!}, \\
&\ \ \ \left( 2(\sqrt{\left \vert x_{1}\right \vert }+...+\sqrt{\left \vert
x_{k}\right \vert })+...+\left \vert x_{r}\right \vert <1\right) .
\end{align*}

In this study primarily we define a new type multivariable hypergeometric
function, then we obtain generating functions. Last, we derive families of
multilinear and multilateral generating functions for the multivariable
hypergeometric functions and their conclusions are also given.


\section{Generating Functions}

In this section, we define a new type multivariable hypergeometric function
and we obtain generating functions for these functions.

\begin{definition}
\label{def:1} We define a new type multivariable hypergeometric function as
below: 
\begin{align}
^{\left( k\right) }E^{\left( r\right) }&\left( {\small \alpha ,\beta }_{k+1}%
{\small ,...,\beta }_{r}{\small ;\gamma }_{1}{\small ,...,\gamma }_{r}%
{\small ;x}_{1}{\small ,...,x}_{r}\right)  \label{b1} \\
:&=\sum \limits_{m_{1},...,m_{r}=0}^{\infty }{\small (\alpha )}_{\rho\left(
m_{1}+...+m_{k}\right) +m_{k+1}+...+m_{r}}\frac{(\beta_{k+1})_{m_{k+1}}...(%
\beta _{r})_{m_{r}}}{(\gamma _{1})_{m_{1}}...(\gamma_{r})_{m_{r}}}\frac{%
x_{1}^{m_{1}}}{m_{1}!}{\small ...}\frac{x_{r}^{m_{r}}}{m_{r}!},  \notag
\end{align}
where $\rho \left( \sqrt{\left \vert x_{1}\right \vert }+...+\sqrt{%
\left\vert x_{k}\right \vert }\right) +\left \vert x_{k+1}\right \vert
+...+\left\vert x_{r}\right \vert <1$, $\rho \geqslant 0.$
\end{definition}

\begin{theorem}
\label{thm:1} We have the following generating function for the
multivariable hypergeometric function $^{\left( k\right) }E^{\left( r\right)
}$ defined by (\ref{b1}): 
\begin{align}
&\sum \limits_{n=0}^{\infty }\frac{(\lambda )_{n}}{n!}\ ^{\left( k\right)}E\
^{\left( r\right) }\left( {\small \lambda +n,\beta }_{k+1}{\small ,...,\beta 
}_{r}{\small ;\gamma }_{1}{\small ,...,\gamma }_{r}{\small ;x}_{1}{\small %
,...,x}_{r}\right) t^{n} =(1-t)^{-\lambda }  \notag  \label{b2} \\
&\ \ \times \text{ }^{\left( k\right) }E\ ^{\left( r\right) }\left( {\small %
\lambda ,\beta }_{k+1}{\small ,...,\beta }_{r}{\small ;\gamma }_{1}{\small %
,...,\gamma }_{r}{\small ;\frac{x_{1}}{\left( 1-t\right) ^{\rho }},...,\frac{%
x_{k}}{\left( 1-t\right) ^{\rho }},\frac{x_{k+1}}{\left( 1-t\right) },...,}%
\frac{x_{r}}{\left( 1-t\right) }\right)  \notag
\end{align}
where $\lambda \in \mathbb{C},\ \ \left \vert t\right \vert <1$.
\end{theorem}

\begin{proof}
Let $T$\ denote the first member of assertion (\ref{b2}),

\begin{equation*}
T=\sum \limits_{n=0}^{\infty }\frac{(\lambda )_{n}}{n!}\ ^{\left(
k\right)}E\ ^{\left( r\right) }\left( {\small \lambda +n,\beta }_{k+1}%
{\small ,...,\beta }_{r}{\small ;\gamma }_{1}{\small ,...,\gamma }_{r}%
{\small ;x}_{1}{\small ,...,x}_{r}\right) t^{n}
\end{equation*}
from Pochhammer symbol properties, 
\begin{align*}
T =&\sum \limits_{m_{1},...,m_{r}=0}^{\infty }\left( \lambda \right)_{\rho
\left( m_{1}+...+m_{k}\right) +...+m_{k+1}+...+m_{r}}\frac{%
(\beta_{k+1})_{m_{k+1}}...(\beta _{r})_{m_{r}}}{(\gamma
_{1})_{m_{1}}...(\gamma _{r})_{m_{r}}} \\
&\times \frac{\left( x_{1}\right) ^{m_{1}}}{m_{1}!}...\frac{%
\left(x_{r}\right) ^{m_{r}}}{m_{r}!}\sum \limits_{n=0}^{\infty }\frac{%
(\lambda+\rho (m_{1}+...+m_{k})+m_{k+1}+...+m_{r})_{n}\ t^{n}}{(n)!} \\
&=(1-t)^{-\lambda }\sum \limits_{m_{1},...,m_{r}=0}^{\infty }\left(
\lambda\right) _{\rho \left( m_{1}+...+m_{k}\right) +...+m_{k+1}+...+m_{r}}%
\frac{(\beta _{k+1})_{m_{k+1}}...(\beta _{r})_{m_{r}}}{%
(\gamma_{1})_{m_{1}}...(\gamma _{r})_{m_{r}}} \\
&\times \left( \frac{x_{1}}{\left( 1-t\right) ^{\rho }}\right)^{m_{1}}...%
\left( \frac{x_{k}}{\left( 1-t\right) ^{\rho }}\right)^{m_{k}}\left( \frac{%
x_{k+1}}{\left( 1-t\right) }\right) ^{m_{k+1}}...\left(\frac{x_{r}}{\left(
1-t\right) }\right) ^{m_{r}}\frac{1}{m_{1}!}...\frac{1}{m_{r}!} \\
&=(1-t)^{-\lambda }\text{ }^{\left( k\right) }E\ ^{\left( r\right) }\left( 
{\small \lambda ,\beta }_{k+1}{\small ,...,\beta }_{r}{\small ;\gamma }_{1}%
{\small ,...,\gamma }_{r}{\small ;\frac{x_{1}}{\left( 1-t\right) ^{\rho }}
,...,\frac{x_{k}}{\left( 1-t\right) ^{\rho }},\frac{x_{k+1}}{%
\left(1-t\right) },...,}\frac{x_{r}}{\left( 1-t\right) }\right)
\end{align*}
which completes the proof.
\end{proof}

If we set $\rho =2$\ in (\ref{b1}), we clearly see that the multivariable
hypergeometric functions are a generalization of the multivariable fourth
kind Horn functions $\ ^{\left( k\right) }H_{4}^{\left( r\right) }$ \cite{H}%
: 
\begin{align*}
^{\left( k\right) }E& ^{\left( r\right) }\left( {\small \alpha ,\beta }_{k+1}%
{\small ,...,\beta }_{r}{\small ;\gamma }_{1}{\small ,...,\gamma }_{r}%
{\small ;x}_{1}{\small ,...,x}_{r}\right) \\
& =\ ^{\left( k\right) }H\ _{4}^{\left( r\right) }\left( {\small \alpha
,\beta }_{k+1}{\small ,...,\beta }_{r}{\small ;\gamma }_{1}{\small %
,...,\gamma }_{r}{\small ;x}_{1}{\small ,...,x}_{r}\right) .
\end{align*}

\begin{corollary}
\label{cor:1} If we take $\rho =2$ in Theorem ~\ref{thm:1}, then we have the
following relation for the multivariable fourth kind Horn functions\ \cite{H}%
: 
\begin{align*}
\sum \limits_{n=0}^{\infty }&\frac{(\lambda )_{n}}{n!}~^{\left( k\right)
}H_{4}^{\left( r\right) }\left( {\small \lambda +n,\beta }_{k+1}{\small %
,...,\beta }_{r}{\small ;\gamma }_{1}{\small ,...,\gamma }_{r}{\small ;x}_{1}%
{\small ,...,x}_{r}\right) t^{n} =(1-t)^{-\lambda }~^{\left( k\right) } \\
& \times H\ _{4}^{\left( r\right) }\left({\small \lambda ,\beta }_{k+1}%
{\small ,...,\beta }_{r}{\small ;\gamma }_{1}{\small ,...,\gamma }_{r}%
{\small ;\frac{x_{1}}{\left( 1-t\right) ^{2}},...,\frac{x_{k}}{\left(
1-t\right) ^{2}},\frac{x_{k+1}}{\left( 1-t\right) },...,}\frac{x_{r}}{\left(
1-t\right) }\right) ,
\end{align*}
where $\lambda \in \mathbb{C}$ and $\left \vert t\right \vert <1$.
\end{corollary}

If we choose $k=1$, $r=2$\ in Corollary ~\ref{cor:1}, we immediately have
the following conclusion for the fourth kind Horn functions \cite{SM}.

\begin{remark}
\label{rem:1} We have 
\begin{align*}
\sum \limits_{n=0}^{\infty }&\frac{(\lambda )_{n}}{n!} H_{4}\left(
\lambda+n,\beta ;\gamma _{1},\gamma _{2};x_{1},x_{2}\right) t^{n}
=(1-t)^{-\lambda } H_{4}\left( \lambda ,\beta ;\gamma _{1},\gamma _{2};\frac{%
x_{1}}{(1-t)^{2}},\frac{x_{2}}{\left( 1-t\right) }\right)
\end{align*}
where $\lambda \in \mathbb{C}$ and $~\left \vert t\right \vert <1$.
\end{remark}

If we set $k=0$\ in Corollary ~\ref{cor:1}, we have the following relation
for the first kind Lauricella functions in \cite{SM}.

\begin{remark}
\label{rem:2} We have 
\begin{align*}
\sum \limits_{n=0}^{\infty }&\frac{(\lambda )_{n}}{n!}~
F_{A}^{\left(r\right) }\left( {\small \lambda +n,\beta }_{1}{\small %
,...,\beta }_{r}{\small ;\gamma }_{1}{\small ,...,\gamma }_{r}{\small ;x}_{1}%
{\small ,...,x}_{r}\right) t^{n} \\
& =(1-t)^{-\lambda } F_{A}^{\left( r\right) }\left( {\small \lambda ,\beta }%
_{1}{\small ,...,\beta }_{r}{\small ;\gamma }_{1}{\small ,...,\gamma }_{r}%
{\small ;\frac{x_{1}}{\left( 1-t\right) },...,}\frac{x_{r}}{\left(1-t\right) 
}\right) ,
\end{align*}
where $\lambda \in \mathbb{C}$ and $\left \vert t\right \vert <1$.
\end{remark}

If we choose $k=0$, $r=2$\ in Corollary \ref{cor:1}, we immediately have the
following conclusion for the second kind Appell functions \cite{SM}.

\begin{remark}
\label{rem:3} We have the following generating function for the second kind
Appell functions: 
\begin{align*}
\sum \limits_{n=0}^{\infty }&\frac{(\lambda )_{n}}{n!} F_{2}\left( \lambda
+n,\beta _{1},\beta _{2};\gamma _{1},\gamma _{2};x_{1},x_{2}\right) t^{n}
=(1-t)^{-\lambda }~F_{2}\left( \lambda ,\beta _{1},\beta
_{2};\gamma_{1},\gamma _{2};\frac{x_{1}}{(1-t)},\frac{x_{2}}{(1-t)}\right) ,
\end{align*}
where $\lambda \in \mathbb{C}$ and $\left \vert t\right \vert <1$.
\end{remark}

\begin{theorem}
\label{thm:2} We have the following generating function for the
multivariable hypergeometric functions $^{\left( k\right) }E^{\left(
r\right) }$ defined by (\ref{b1}): 
\begin{align}
&\sum \limits_{n=0}^{\infty }\frac{(\lambda )_{n}}{n!}~^{\left( k\right)}E\
^{\left( r\right) }\left( {\small -n,\beta }_{k+1}{\small ,...,\beta }_{r}%
{\small ;\gamma }_{1}{\small ,...,\gamma }_{r}{\small ;x}_{1}{\small ,...,x}%
_{r}\right) t^{n} =(1-t)^{-\lambda }  \label{b3} \\
&\times \text{ }^{\left( k\right) }E\ ^{\left( r\right) }\left( {\small %
\lambda ,\beta }_{k+1}{\small ,...,\beta }_{r}{\small ;\gamma }_{1}{\small %
,...,\gamma }_{r}{\small ;\frac{x_{1}\left( -t\right) ^{\rho }}{%
\left(1-t\right) ^{\rho }},...,\frac{x_{k}\left( -t\right) ^{\rho }}{%
\left(1-t\right) ^{\rho }},\frac{-x_{k+1}t}{\left( 1-t\right) },...,}\frac{%
-x_{r}t}{\left( 1-t\right) }\right)  \notag
\end{align}
where $\lambda \in \mathbb{C}$, \ $\left \vert t\right \vert <1$.
\end{theorem}

\begin{proof}
Let $S$\ denote the first member of assertion (\ref{b3}). Then, 
\begin{align*}
S =&\sum \limits_{n=0}^{\infty }\sum \limits_{m_{1},...,m_{r}=0}^{\rho
(m_{1}+...+m_{k})+m_{k+1}+...+m_{r}\leq n}(\lambda )_{n}(-n)_{\rho
(m_{1}+...+m_{k})+m_{k+1}+...+m_{r}} \\
&\ \ \ \times \frac{(\beta _{k+1})_{m_{k+1}}...(\beta_{r})_{m_{r}}}{(\gamma
_{1})_{m_{1}}...(\gamma _{r})_{m_{r}}}\frac{\left(x_{1}\right) ^{m_{1}}}{%
m_{1}!}...\frac{\left( x_{r}\right) ^{m_{r}}}{m_{r}!}\frac{t^{n}}{n!}
\end{align*}
where we have used the relation \cite[p.102]{SM} 
\begin{align}
\sum \limits_{n=0}^{\infty
}&\sum\limits_{k_{1},...,k_{r}=0}^{m_{1}k_{1}+...+m_{r}k_{r}\leq n}\Phi
\left(k_{1},...,k_{r};n\right) =\sum \limits_{n=0}^{\infty }\sum
\limits_{k_{1},...,k_{r}=0}^{\infty}\Phi
\left(k_{1},...,k_{r};n+m_{1}k_{1}+...+m_{r}k_{r}\right) ,  \label{b4}
\end{align}
and we get 
\begin{align*}
S &=\sum \limits_{m_{1},...,m_{r}=0}^{\infty }\frac{%
(\beta_{k+1})_{m_{k+1}}...(\beta _{r})_{m_{r}}}{(\gamma
_{1})_{m_{1}}...(\gamma_{r})_{m_{r}}} \\
&\times \frac{(x_{1}\left( -t\right) ^{\rho
})^{m_{1}}...(x_{k}\left(-t\right) ^{\rho
})^{m_{k}}(-x_{k+1}t)^{m_{k+1}}...(-x_{r}t)^{m_{r}}}{\left(m_{1}\right)
!...\left( m_{k}\right) !\left( m_{k+1}\right) !...\left(m_{r}\right) !} \\
&\times \sum \limits_{n=0}^{\infty }\frac{(\lambda
)_{n+\rho(m_{1}+...+m_{k})+m_{k+1}+...+m_{r}}t^{n}}{(n)!} \\
&=\sum \limits_{m_{1},...,m_{r}=0}^{\infty }(\lambda
)_{\rho(m_{1}+...+m_{k})+m_{k+1}+...+m_{r}}\frac{(\beta
_{k+1})_{m_{k+1}}...(\beta_{r})_{m_{r}}}{(\gamma _{1})_{m_{1}}...(\gamma
_{r})_{m_{r}}} \\
&\times \frac{(x_{1}\left( -t\right) ^{\rho
})^{m_{1}}...(x_{k}\left(-t\right) ^{\rho
})^{m_{k}}(-x_{k+1}t)^{m_{k+1}}...(-x_{r}t)^{m_{r}}}{\left(m_{1}\right)
!...\left( m_{k}\right) !\left( m_{k+1}\right) !...\left(m_{r}\right) !} \\
&\times \sum \limits_{n=0}^{\infty }\frac{(\lambda
+\rho(m_{1}+...+m_{k})+m_{k+1}+...+m_{r})_{n}\ t^{n}}{(n)!} \\
&=(1-t)^{-\lambda }\sum \limits_{m_{1},...,m_{r}=0}^{\infty }\left(
\lambda\right) _{\rho \left( m_{1}+...+m_{k}\right) +...+m_{k+1}+...+m_{r}}%
\frac{(\beta _{k+1})_{m_{k+1}}...(\beta _{r})_{m_{r}}}{%
(\gamma_{1})_{m_{1}}...(\gamma _{r})_{m_{r}}} \\
& \times \left( \frac{x_{1}\left( -t\right) ^{\rho }}{\left(
1-t\right)^{\rho }}\right) ^{m_{1}}...\left( \frac{x_{k}\left( -t\right)
^{\rho }}{\left( 1-t\right) ^{\rho }}\right) ^{m_{k}}\left( \frac{-x_{k+1}t}{%
\left(1-t\right) }\right) ^{m_{k+1}}...\left( \frac{-x_{r}t}{\left(
1-t\right) }\right) ^{m_{r}}\frac{1}{m_{1}!}...\frac{1}{m_{r}!} \\
&=(1-t)^{-\lambda }\text{ } \\
&\times\ ^{\left( k\right) }E\ ^{\left( r\right) }\left({\small \lambda
,\beta }_{k+1}{\small ,...,\beta }_{r}{\small ;\gamma }_{1}{\small %
,...,\gamma }_{r}{\small ;\frac{x_{1}\left( -t\right) ^{\rho }}{\left(
1-t\right) ^{\rho }},...,\frac{x_{k}\left( -t\right) ^{\rho }}{%
\left(1-t\right) ^{\rho }},\frac{-x_{k+1}t}{\left( 1-t\right) },...,}\frac{%
-x_{r}t}{\left( 1-t\right) }\right)
\end{align*}
which completes the proof.
\end{proof}

\begin{theorem}
\label{thm:3} We have the following generating function for the
multivariable hypergeometric functions $\ ^{\left( k\right) }E\ ^{\left(
r\right) }\ $ defined by (\ref{b1}): 
\begin{align}
\sum\limits_{n=0}^{\infty }& \ ^{\left( k\right) }E\ ^{\left( r\right)
}\left( {\small -n,\beta }_{k+1}{\small ,...,\beta }_{r}{\small ;\gamma }_{1}%
{\small ,...,\gamma }_{r}{\small ;x}_{1}{\small ,...,x}_{r}\right) \frac{%
t^{n}}{n!}  \notag \\
& =e^{t}{}_{0}F_{1}\left( -,\gamma _{1};x_{1}\left( -t\right) ^{\rho
}\right) ...{}_{0}F_{1}\left( -,\gamma _{k};x_{k}\left( -t\right) ^{\rho
}\right)   \notag \\
& \ \ \ \ \times \Phi \left( \beta _{k+1},\gamma _{k+1};-x_{k+1}t\right)
...\Phi \left( \beta _{r},\gamma _{r};-x_{r}t\right) .  \label{b5}
\end{align}%
where $\left\vert t\right\vert <1$\ and $_{0}F_{1}$\ is hypergeometric
series, $\Phi $\ is confluent hypergeometric function.
\end{theorem}

\begin{proof}
Let $S$\ denote the first member of assertion (\ref{b5}). Then, 
\begin{align*}
S& =\sum\limits_{n=0}^{\infty }\ ^{\left( k\right) }E\ ^{\left( r\right)
}\left( {\small -n,\beta }_{k+1}{\small ,...,\beta }_{r}{\small ;\gamma }_{1}%
{\small ,...,\gamma }_{r}{\small ;x}_{1}{\small ,...,x}_{r}\right) \frac{%
t^{n}}{n!} \\
& =\sum\limits_{n=0}^{\infty }\sum\limits_{m_{1},...,m_{r}=0}^{\rho
(m_{1}+...+m_{k})+m_{k+1}+...+m_{r}\leq n}(-n)_{\rho
(m_{1}+...+m_{k})+m_{k+1}+...+m_{r}} \\
& \ \ \ \ \ \ \ \ \ \ \times \frac{(\beta _{k+1})_{m_{k+1}}...(\beta
_{r})_{m_{r}}}{(\gamma _{1})_{m_{1}}...(\gamma _{r})_{m_{r}}}\frac{\left(
x_{1}\right) ^{m_{1}}}{m_{1}!}...\frac{\left( x_{r}\right) ^{m_{r}}}{m_{r}!}%
\frac{t^{n}}{n!}
\end{align*}%
and use (\ref{b4}), 
\begin{align*}
S& =\sum\limits_{m_{1},...,m_{r}=0}^{\infty }\frac{(\beta
_{k+1})_{m_{k+1}}...(\beta _{r})_{m_{r}}}{(\gamma _{1})_{m_{1}}...(\gamma
_{r})_{m_{r}}} \\
& \times \frac{(x_{1}\left( -t\right) ^{\rho })^{m_{1}}...(x_{k}\left(
-t\right) ^{\rho })^{m_{k}}(-x_{k+1}t)^{m_{k+1}}...(-x_{r}t)^{m_{r}}}{\left(
m_{1}\right) !...\left( m_{k}\right) !\left( m_{k+1}\right) !...\left(
m_{r}\right) !}\sum\limits_{n=0}^{\infty }\frac{t^{n}}{n!} \\
& =e^{t}\sum\limits_{m_{1},...,m_{r}=0}^{\infty }\frac{(\beta
_{k+1})_{m_{k+1}}...(\beta _{r})_{m_{r}}}{(\gamma _{1})_{m_{1}}...(\gamma
_{r})_{m_{r}}}\frac{(x_{1}\left( -t\right) ^{\rho })^{m_{1}}}{\left(
m_{1}\right) !}...\frac{(x_{k}\left( -t\right) ^{\rho })^{m_{k}}}{\left(
m_{k}\right) !}\frac{(-x_{k+1}t)^{m_{k+1}}}{\left( m_{k+1}\right) !}...\frac{%
(-x_{r}t)^{m_{r}}}{\left( m_{r}\right) !} \\
& =e^{t}{}_{0}F_{1}\left( -,\gamma _{1};x_{1}\left( -t\right) ^{\rho
}\right) ...{}_{0}F_{1}\left( -,\gamma _{k};x_{k}\left( -t\right) ^{\rho
}\right) \Phi \left( \beta _{k+1},\gamma _{k+1};-x_{k+1}t\right) ...\Phi
\left( \beta _{r},\gamma _{r};-x_{r}t\right) ,
\end{align*}%
which completes the proof.
\end{proof}

\begin{theorem}
\label{thm:4} We have the following generating function for the
multivariable hypergeometric functions $~^{\left( k\right) }E\ ^{\left(
r\right) }$ defined by (\ref{b1}): 
\begin{align}
\sum\limits_{n=0}^{\infty }& \frac{(\lambda )_{n}}{n!}~^{\left( k\right)
}E^{\left( r\right) }\left( {\small \alpha ,\beta }_{k+1}{\small ,...,\beta }%
_{r}{\small ;1-\lambda -n,\gamma }_{2}{\small ,...,\gamma }_{r}{\small ;x}%
_{1}{\small ,...,x}_{r}\right) t^{n}  \label{b6} \\
& =(1-t)^{-\lambda }\text{ }^{\left( k\right) }E\ ^{\left( r\right) }\left( 
{\small \alpha ,\beta }_{k+1}{\small ,...,\beta }_{r}{\small ;1-\lambda
,\gamma }_{2}{\small ,...,\gamma }_{r}{\small ;x}_{1}\left( 1-t\right) 
{\small ,x}_{2},{\small ...,x}_{r}\right)   \notag
\end{align}%
where $\lambda \in \mathbb{C},\ \ \left\vert t\right\vert <1$.
\end{theorem}

\begin{proof}
Let $T$\ denote the first member of assertion (\ref{b6}). Then, 
\begin{align*}
T& =\sum\limits_{n=0}^{\infty }\sum\limits_{m_{1},...,m_{r}=0}^{\infty
}(\lambda )_{n}(\alpha )_{\rho (m_{1}+...+m_{k})+m_{k+1}+...+m_{r}} \\
& \ \ \ \ \ \ \ \ \times \frac{(\beta _{k+1})_{m_{k+1}}...(\beta
_{r})_{m_{r}}}{\left( 1-\lambda -n\right) _{m_{1}}(\gamma
_{2})_{m_{2}}...(\gamma _{r})_{m_{r}}}\frac{\left( x_{1}\right) ^{m_{1}}}{%
m_{1}!}...\frac{\left( x_{r}\right) ^{m_{r}}}{m_{r}!}\frac{t^{n}}{n!} \\
& =\sum\limits_{m_{1},...,m_{r}=0}^{\infty }(\alpha )_{\rho
(m_{1}+...+m_{k})+m_{k+1}+...+m_{r}}\frac{(\beta _{k+1})_{m_{k+1}}...(\beta
_{r})_{m_{r}}}{\left( 1-\lambda \right) _{m_{1}}(\gamma
_{2})_{m_{2}}...(\gamma _{r})_{m_{r}}} \\
& \times \frac{\left( x_{1}\right) ^{m_{1}}}{m_{1}!}...\frac{\left(
x_{r}\right) ^{m_{r}}}{m_{r}!}\sum\limits_{n=0}^{\infty }\frac{\left(
\lambda -m_{1}\right) _{n}}{n!}t^{n} \\
& =\left( 1-t\right) ^{-\lambda }\sum\limits_{m_{1},...,m_{r}=0}^{\infty
}(\alpha )_{\rho (m_{1}+...+m_{k})+m_{k+1}+...+m_{r}}\frac{(\beta
_{k+1})_{m_{k+1}}...(\beta _{r})_{m_{r}}}{\left( 1-\lambda \right)
_{m_{1}}(\gamma _{2})_{m_{2}}...(\gamma _{r})_{m_{r}}} \\
& \times \frac{\left( x_{1}\left( 1-t\right) \right) ^{m_{1}}}{m_{1}!}...%
\frac{\left( x_{r}\right) ^{m_{r}}}{m_{r}!} \\
& =(1-t)^{-\lambda }\text{ }^{\left( k\right) }E\ ^{\left( r\right) }\left( 
{\small \alpha ,\beta }_{k+1}{\small ,...,\beta }_{r}{\small ;1-\lambda
,\gamma }_{2}{\small ,...,\gamma }_{r}{\small ;x}_{1}\left( 1-t\right) 
{\small ,x}_{2},{\small ...,x}_{r}\right) 
\end{align*}%
which completes the proof.
\end{proof}

In the next theorem, let $\Psi _{m}$\ denote the following special functions 
$^{\left( k\right) }E\ ^{\left( r\right) }.$ 
\begin{equation}
\Psi _{m}=\ ^{\left( k\right) }E\ ^{\left( r\right) }\left( {\small %
\alpha,\beta }_{k+1}{\small ,...,\beta }_{r}{\small ;1-\lambda -m,\gamma }%
_{2}{\small ,...,\gamma }_{r}{\small ;x}_{1}{\small ,...,x}_{r}\right) .
\label{b7}
\end{equation}

\begin{theorem}
\label{thm:5} The following generating function for the multivariable
hypergeometric functions $\Psi _{m}$\ holds true: 
\begin{align}
\sum \limits_{n=0}^{\infty }&\binom{\lambda +m+n-1}{n}~\Psi
_{m+n}\left(x_{1},...,x_{r}\right) t^{n} =(1-t)^{-\lambda -m}\ \Psi
_{m}\left(x_{1}\left( 1-t\right) ,x_{2},...,x_{r}\right) .  \label{b8}
\end{align}
\end{theorem}

\begin{proof}
Let $T$\ denote the first member of assertion (\ref{b8}). Then, 
\begin{align*}
T &=\sum \limits_{n=0}^{\infty }\binom{\lambda +m+n-1}{n}~\Psi_{m+n}\left(
x_{1},...,x_{r}\right) t^{n} \\
&=\sum \limits_{n=0}^{\infty }\binom{\lambda +m+n-1}{n}\ ^{\left(
k\right)}E\ ^{\left( r\right) }\left( {\small \alpha ,\beta }_{k+1}{\small %
,...,\beta }_{r}{\small ;1-\lambda -m-n,\gamma }_{2}{\small ,...,\gamma }_{r}%
{\small ;x}_{1}{\small ,...,x}_{r}\right) t^{n} \\
\end{align*}
from Theorem ~\ref{thm:4}, 
\begin{align*}
T &=(1-t)^{-\lambda -m}\text{ }^{\left( k\right) }E\ ^{\left(
r\right)}\left( {\small \alpha ,\beta }_{k+1}{\small ,...,\beta }_{r}{\small %
;1-\lambda -m,\gamma }_{2}{\small ,...,\gamma }_{r}{\small ;x}%
_{1}\left(1-t\right) {\small ,x}_{2},{\small ...,x}_{r}\right) \\
&=(1-t)^{-\lambda -m}\ \Psi _{m}\left( x_{1}\left(
1-t\right),x_{2},...,x_{r}\right) ,
\end{align*}
which completes the proof.
\end{proof}


Similar to Corollary 1, we also get the following corollary for Theorems 2,
3, 4 and 5.

\begin{corollary}
\label{cor:2}\textbf{\ }If we take $\rho =2$ in Theorems 2, 3, 4 and 5, then
we have new generating relations for the multivariable fourth kind Horn
functions. \newline If we take $\rho =2,~k=1~and~r=2$ in
Theorems 2, 3, 4 and 5, then we have new generating relations for the fourth
kind Horn functions. \newline  If we take $\rho =2,~k=0~$ in Theorems
2, 3, 4 and 5, then we have new generating relations for the Lauricella
functions.\newline If we take\ $\rho =2,~k=0~and~r=2$ in Theorems 2, 3, 4 and 5, then
we have new generating relations for the second kind Appell functions.
\end{corollary}

\section{Multilinear and Multilateral Generating Functions}

In this section, we derive several families of multilinear and multilateral
generating function for multivariable hypergeometric functions defined by (%
\ref{b1}) by using the similar method considered in \cite{E-S, R-E}.

\begin{theorem}
\label{thm:6} Corresponding to an identically non-vanishing function $\Omega
_{\mu}(y_{1},...,y_{s}\,)$\ of $s$\ complex variables $y_{1},...,y_{s} $\ $%
(s\in N) $\ and of complex order $\mu $, let 
\begin{equation*}
\Lambda _{\mu ,\psi }(y_{1},...,y_{s};\zeta ):=\sum
\limits_{k=0}^{\infty}a_{k}\Omega _{\mu +\psi k}(y_{1},...,y_{s})\zeta ^{k}
\end{equation*}
where $a_{k}\neq 0\,,\, \, \mu ,\psi \in C$\ and 
\begin{align*}
\Theta _{n,p}^{\mu ,\psi }&\left( x_{1},x_{2};y_{1},...,y_{s};\xi \right) \\
&:=\sum \limits_{k=0}^{\left[ n/p\right] }a_{k}\left( \lambda
\right)_{n-pk}\ \ ^{\left( k\right) }E\ ^{\left( r\right) }\left( {\small %
\lambda+n-pk,\beta }_{k+1}{\small ,...,\beta }_{r}{\small ;\gamma }_{1}%
{\small ,...,\gamma }_{r}{\small ;x}_{1}{\small ,...,x}_{r}\right) \Omega
_{\mu +\psi k}(y_{1},...,y_{s})\frac{\xi ^{k}}{(n-pk)!}.
\end{align*}
Then, for $p\in N,$ we have 
\begin{align}
\sum \limits_{n=0}^{\infty }&\Theta _{n,p}^{\mu ,\psi
}\left(x_{1},x_{2};y_{1},...,y_{s};\frac{\eta }{t^{p}}\right) t^{n}=\Lambda
_{\mu,\psi }(y_{1},...,y_{s};\eta )(1-t)^{-\lambda }  \notag \\
&\times \text{ }^{\left( k\right) }E\ ^{\left( r\right) }\left( {\small %
\lambda ,\beta }_{k+1}{\small ,...,\beta }_{r}{\small ;\gamma }_{1}{\small %
,...,\gamma }_{r}{\small ;\frac{x_{1}}{\left( 1-t\right) ^{\rho }},...,\frac{%
x_{k}}{\left( 1-t\right) ^{\rho }},\frac{x_{k+1}}{\left( 1-t\right) },...,}%
\frac{x_{r}}{\left( 1-t\right) }\right)  \label{c1}
\end{align}
provided that each member of (\ref{c1}) exists.
\end{theorem}

\begin{proof}
For convenience, let $S$\ denote the first member of the assertion (\ref{c1}%
). Then, 
\begin{align*}
S &=\sum \limits_{n=0}^{\infty }\sum \limits_{k=0}^{[n/p]}\
a_{k}\left(\lambda \right) _{n-pk}\ ^{\left( k\right) }E\ ^{\left( r\right)
}\left({\small \lambda +n-pk,\beta }_{k+1}{\small ,...,\beta }_{r}{\small %
;\gamma }_{1}{\small ,...,\gamma }_{r}{\small ;x}_{1}{\small ,...,x}%
_{r}\right) \\
&\times \Omega _{\mu +\psi k}(y_{1},...,y_{s})\eta ^{k}\frac{t^{n-pk}}{%
(n-pk)!}.
\end{align*}
Replacing $n$\ by $n+pk,$\ we may write that 
\begin{align*}
S &=\sum \limits_{n=0}^{\infty }\sum \limits_{k=0}^{\infty
}a_{k}\left(\lambda \right) _{n}\ ^{\left( k\right) }E\ ^{\left( r\right)
}\left( {\small \lambda +n,\beta }_{k+1}{\small ,...,\beta }_{r}{\small %
;\gamma }_{1}{\small ,...,\gamma }_{r}{\small ;x}_{1}{\small ,...,x}%
_{r}\right) \\
&\times \Omega _{\mu +\psi k}(y_{1},...,y_{s})\eta ^{k}\frac{t^{n}}{n!} \\
&=\sum \limits_{n=0}^{\infty }\left( \lambda \right) _{n}\ ^{\left(k\right)
}E\ ^{\left( r\right) }\left( {\small \lambda +n,\beta }_{k+1}{\small %
,...,\beta }_{r}{\small ;\gamma }_{1}{\small ,...,\gamma }_{r}{\small ;x}_{1}%
{\small ,...,x}_{r}\right) \frac{t^{n}}{n!} \\
&\times \sum \limits_{k=0}^{\infty }a_{k}\Omega _{\mu +\psi
k}(y_{1},...,y_{r})\eta ^{k} \\
&=\text{ }^{\left( k\right) }E\ ^{\left( r\right) }\left( {\small %
\lambda,\beta }_{k+1}{\small ,...,\beta }_{r}{\small ;\gamma }_{1}{\small %
,...,\gamma }_{r}{\small ;\frac{x_{1}}{\left( 1-t\right) ^{\rho }},...,\frac{%
x_{k}}{\left( 1-t\right) ^{\rho }},\frac{x_{k+1}}{\left( 1-t\right) },...,}%
\frac{x_{r}}{\left( 1-t\right) }\right) \\
&\times (1-t)^{-\lambda }\ \Lambda _{\mu ,\psi }(y_{1},...,y_{s};\eta ),
\end{align*}
which completes the proof.
\end{proof}

In a similar manner, we also get the Theorem \ref{thm:7} and \ref{thm:8}
immediately.

\begin{theorem}
\label{thm:7} Corresponding to an identically non-vanishing function $\Omega
_{\mu}(y_{1},...,y_{s}\,)$\ of $s$\ complex variables $y_{1},...,y_{s} $\ $%
(s\in N) $\ and of complex order $\mu $, let 
\begin{equation*}
\Lambda _{\mu ,\psi }(y_{1},...,y_{s};\zeta ):=\sum
\limits_{k=0}^{\infty}a_{k}\Omega _{\mu +\psi k}(y_{1},...,y_{s})\zeta ^{k}
\end{equation*}
where $a_{k}\neq 0\,,\, \, \mu ,\psi \in C$ and 
\begin{align*}
\Theta _{n,p}^{\mu ,\psi }&\left( x_{1},x_{2};y_{1},...,y_{s};\xi \right) \\
&:=\sum \limits_{k=0}^{[n/p]}a_{k}\left( \lambda \right) _{n-pk}\
~^{\left(k\right) }E\ ^{\left( r\right) }\left( {\small -n-pk,\beta }_{k+1}%
{\small ,...,\beta }_{r}{\small ;\gamma }_{1}{\small ,...,\gamma }_{r}%
{\small ;x}_{1}{\small ,...,x}_{r}\right)\Omega _{\mu +\psi
k}(y_{1},...,y_{s})\frac{\xi ^{k}}{(n-pk)!}.
\end{align*}
Then, for \ $p\in N,$ we have 
\begin{align}
\sum \limits_{n=0}^{\infty }&\Theta _{n,p}^{\mu ,\psi
}\left(x_{1},x_{2};y_{1},...,y_{s};\frac{\eta }{t^{p}}\right) t^{n}=\Lambda
_{\mu,\psi }(y_{1},...,y_{s};\eta )(1-t)^{-\lambda }  \label{c2} \\
&\times \ ^{\left( k\right) }E\ ^{\left( r\right) }\left( {\small %
\lambda,\beta }_{k+1}{\small ,...,\beta }_{r}{\small ;\gamma }_{1}{\small %
,...,\gamma }_{r}{\small ;\frac{x_{1}\left( -t\right) ^{\rho }}{%
\left(1-t\right) ^{\rho }},...,\frac{x_{k}\left( -t\right) ^{\rho }}{%
\left(1-t\right) ^{\rho }},\frac{-x_{k+1}t}{\left( 1-t\right) },...,}\frac{%
-x_{r}t}{\left( 1-t\right) }\right) ,  \notag
\end{align}
provided that each member of (\ref{c2}) exists.
\end{theorem}

\begin{theorem}
\label{thm:8} Corresponding to an identically non-vanishing function $\Omega
_{\mu }(y_{1},...,y_{s}\,)$\ of $s$\ complex variables $y_{1},...,y_{s}$ $%
(s\in N)$ and of complex order $\mu $, let 
\begin{equation*}
\Lambda _{\mu ,\psi }(y_{1},...,y_{s};\zeta ):=\sum\limits_{k=0}^{\infty
}a_{k}\Omega _{\mu +\psi k}(y_{1},...,y_{s})\zeta ^{k}
\end{equation*}%
where $a_{k}\neq 0\,,\,\,\mu ,\psi \in C$ and 
\begin{align*}
\Theta _{n,p}^{\mu ,\psi }& \left( x_{1},x_{2};y_{1},...,y_{s};\xi \right) \\
& :=\sum\limits_{k=0}^{[n/p]}a_{k}\ ^{\left( k\right) }E\ ^{\left( r\right)
}\left( {\small -n-pk,\beta }_{k+1}{\small ,...,\beta }_{r}{\small ;\gamma }%
_{1}{\small ,...,\gamma }_{r}{\small ;x}_{1}{\small ,...,x}_{r}\right) \\
& \times \Omega _{\mu +\psi k}(y_{1},...,y_{s})\frac{\xi ^{k}}{(n-pk)!}.
\end{align*}%
Then, for \ $p\in N,$ we have 
\begin{align}
\sum\limits_{n=0}^{\infty }& \Theta _{n,p}^{\mu ,\psi }\left(
x_{1},x_{2};y_{1},...,y_{s};\frac{\eta }{t^{p}}\right) t^{n}=\Lambda _{\mu
,\psi }(y_{1},...,y_{s};\eta )  \notag \\
& \times e^{t}{}_{0}F_{1}\left( -,\gamma _{1};x_{1}\left( -t\right) ^{\rho
}\right) ...{}_{0}F_{1}\left( -,\gamma _{k};x_{k}\left( -t\right) ^{\rho
}\right)  \notag \\
& \times \Phi \left( \beta _{k+1},\gamma _{k+1};-x_{k+1}t\right) ...\Phi
\left( \beta _{r},\gamma _{r};-x_{r}t\right) ,  \label{c3}
\end{align}%
provided that each member of (\ref{c3}) exists.
\end{theorem}

\begin{theorem}
\label{thm:9} Corresponding to an identically non-vanishing function $\Omega
_{\mu }(y_{1},...,y_{r}\,)$\ of $r$\ complex variables $y_{1},...,y_{r}$\ $%
(r\in N)$\ and of complex order $\mu $, let
\end{theorem}

\begin{equation*}
\Lambda _{m,q}(x_{1},...,x_{r};y_{1},...,y_{r};t) :=\sum
\limits_{n=0}^{\infty }a_{n}\ \Psi _{m+nq}\left(x_{1},...,x_{r}\right)
\Omega _{\mu +pn}(y_{1},...,y_{r})t^{n},
\end{equation*}
where $(a_{n}\neq 0\,,\, \, \mu \in C)$, $\Psi _{m}$ is defined by (\ref{b7}%
) and 
\begin{equation*}
N_{n,m,q}^{p,\mu }(y_{1},...,y_{r};z):=\sum_{k=0}^{\left[ n/q\right]%
}a_{k}\Omega _{\mu +pk}(y_{1},...,y_{r})z^{k}.
\end{equation*}
Then, for every nonnegative integer $m$, 
\begin{align}
\sum \limits_{n=0}^{\infty }&\binom{\lambda +m+n-1}{n}\Psi
_{m+n}\left(_{1},...,x_{r}\right) N_{n,m,q}^{p,\mu }(y_{1},...,y_{r};z)t^{n}
\label{c4} \\
&=(1-t)^{-\lambda -m}{}\Lambda _{m,q}\left( x_{1}\left(
1-t\right),x_{2},...,x_{r};y_{1},...,y_{r};\frac{zt^{q}}{\left( 1-t\right)
^{q}}\right)  \notag
\end{align}
provided that each member of (\ref{c4}) exists.

\begin{proof}
Let $T$\ denote the left-hand side of equality (\ref{c4}). Then we have 
\begin{align*}
T &=\sum \limits_{n=0}^{\infty }\binom{\lambda +m+n-1}{n}\ \Psi_{m+n}\left(
x_{1},...,x_{r}\right) \sum_{k=0}^{\left[ n/q\right]}a_{k}\Omega _{\mu
+pk}(y_{1},...,y_{r})z^{k}t^{n} \\
&=\sum \limits_{k=0}^{\infty }\left( \sum_{n=0}^{\infty }\binom{%
\lambda+m+n+qk-1}{n}\ \Psi _{m+n+qk}\left( x_{1},...,x_{r}\right)
t^{n}\right) \\
&\ \ \ \ \ \ \ \ \ \ \ \ \ \ \ \ \ \times a_{k}\Omega
_{\mu+pk}(y_{1},...,y_{r})\left( zt^{q}\right) ^{k} \\
&=\sum \limits_{k=0}^{\infty }\left( 1-t\right) ^{-\lambda
-m-qk}\Psi_{m+qk}\left( x_{1}\left( 1-t\right) ,x_{2},...,x_{r}\right)
a_{k}\Omega_{\mu +pk}(y_{1},...,y_{r})\left( zt^{q}\right) ^{k} \\
&=(1-t)^{-\lambda -m}{}\sum \limits_{k=0}^{\infty }a_{k}\Psi
_{m+qk}\left(x_{1}\left( 1-t\right) ,x_{2},...,x_{r}\right) \Omega
_{\mu+pk}(y_{1},...,y_{r})\left( \frac{zt^{q}}{\left( 1-t\right) ^{q}}%
\right) ^{k} \\
&=(1-t)^{-\lambda -m}\ \Lambda _{m,q}\left( x_{1}\left(
1-t\right),x_{2},...,x_{r};y_{1},...,y_{r};\frac{zt^{q}}{\left( 1-t\right)
^{q}}\right) ,
\end{align*}
which completes the proof.
\end{proof}


\section{Special Cases}

As an application of the above Theorems, when the multivariable function $%
\Omega _{\mu +\psi k}(y_{1},...,y_{s})\,,$\ $k\in N_{0},\,s\in N$, is
expressed in terms of simpler functions of one and more variables, then we
can give further applications of the above theorems. We first set 
\begin{equation*}
s=r,\ \ \text{\ }\Omega _{\mu +\psi k}(y_{1},...,y_{r}\,)=u_{\mu +\psi
k}^{(\alpha _{1},...,\alpha _{r})}(y_{1},...,y_{r})
\end{equation*}
in Theorem \ref{thm:6}, where the Erkus-Srivastava polynomials $%
u_{n}^{(\alpha_{1},...,\alpha _{r})}(x_{1},...,x_{r})$, generated by \cite%
{E-S} 
\begin{equation}
\sum \limits_{n=0}^{\infty }u_{n}^{(\alpha
_{1},...,\alpha_{r})}(x_{1},...,x_{r})\frac{t^{n}}{n!}=\overset{r}{\underset{%
j=1}{\prod }}\ \{(1-x_{j}t^{m_{j}})^{-\alpha _{j}}\}.  \label{d1}
\end{equation}

We are thus led to the following result which provides a class of bilateral
generating functions for the $^{\left( k\right) }E^{\left( r\right) }$
multivariable hypergeometric functions and Erkus-Srivastava polynomials
respectively defined by (\ref{b1}) and generated by (\ref{d1}).

\begin{corollary}
\label{cor:5} If 
\begin{align*}
\Lambda _{\mu ,\psi }(y_{1},...,y_{r};\zeta ) &:=\sum
\limits_{k=0}^{\infty}a_{k}u_{\mu +\psi k}^{(\alpha
_{1},...,\alpha_{r})}(y_{1},...,y_{r})\zeta^{k} \\
(a_{k} \neq &0,\, \, \mu ,\psi \in \mathbb{C})
\end{align*}
then, we have 
\begin{align}
\sum \limits_{n=0}^{\infty }&\sum \limits_{k=0}^{[n/p]}a_{k}\ \left(
\lambda\right) _{n-pk}\ ^{\left( k\right) }E\ ^{\left( r\right) }\left( 
{\small \lambda +n-pk,\beta }_{k+1}{\small ,...,\beta }_{r}{\small ;\gamma }%
_{1}{\small ,...,\gamma }_{r}{\small ;x}_{1}{\small ,...,x}_{r}\right) 
\notag \\
&\times ~u_{\mu +\psi k}^{(\alpha _{1},...,\alpha _{r})}(y_{1},...,y_{r})%
\frac{\eta ^{k}}{t^{pk}}\frac{t^{n}}{(n-pk)!}  \notag \\
&=~^{\left( k\right) }E^{\left( r\right) }\left( {\small \lambda ,\beta }%
_{k+1}{\small ,...,\beta }_{r}{\small ;\gamma }_{1}{\small ,...,\gamma }_{r}%
{\small ;\frac{x_{1}}{\left( 1-t\right) ^{\rho }},...,\frac{x_{k}}{\left(
1-t\right) ^{\rho }},\frac{x_{k+1}}{\left(1-t\right) },...,}\frac{x_{r}}{%
\left( 1-t\right) }\right)  \notag \\
&\times (1-t)^{-\lambda } ~\Lambda _{\mu ,\psi }(y_{1},...,y_{r};\eta ),
\label{d2}
\end{align}
provided that each member of (\ref{d2}) exists.
\end{corollary}

\begin{remark}
\label{rem:13} Using the generating relation (\ref{d1}) for Erkus-Srivastava
polynomials and getting $a_{k}=1,$\ $\mu =0,$\ $\psi =1$\ in Corollary \ref%
{cor:5}, we find that 
\begin{align*}
\sum \limits_{n=0}^{\infty }&\sum \limits_{k=0}^{[n/p]}\left( \lambda\right)
_{n-pk}\ \ ^{\left( k\right) }E\ ^{\left( r\right) }\left( {\small \lambda
+n-pk,\beta }_{k+1}{\small ,...,\beta }_{r}{\small ;\gamma }_{1}{\small %
,...,\gamma }_{r}{\small ;x}_{1}{\small ,...,x}_{r}\right) \\
&\times \ u_{k}^{(\alpha _{1},...,\alpha_{r})}(y_{1},...,y_{r})\dfrac{%
\eta^{k}t^{n-pk}}{(n-pk)!} \\
&=(1-t)^{-\lambda }~^{\left( k\right) }E\ ^{\left( r\right) }\left( {\small %
\lambda ,\beta }_{k+1}{\small ,...,\beta }_{r}{\small ;\gamma }_{1}{\small %
,...,\gamma }_{r}{\small ;\frac{x_{1}}{\left( 1-t\right) ^{\rho }},...,\frac{%
x_{k}}{\left( 1-t\right) ^{\rho }},\frac{x_{k+1}}{\left(1-t\right) },...,}%
\frac{x_{r}}{\left( 1-t\right) }\right) \\
&\times \overset{r}{\underset{j=1}{\prod }}\ \{(1-y_{i}\eta
^{m_{j}})^{-\alpha _{i}}\}.
\end{align*}
\end{remark}

If we set 
\begin{equation*}
\Omega _{\mu +\psi k}(y_{1},...,y_{r}\,)=\ ^{\left( k\right) }E\ ^{\left(
r\right) }\left( {\small \lambda +(\mu +\psi k),\beta }_{k+1}{\small %
,...,\beta }_{r}{\small ;\gamma }_{1}{\small ,...,\gamma }_{r}{\small ;x}_{1}%
{\small ,...,x}_{r}\right).
\end{equation*}
or 
\begin{equation*}
\Omega _{\mu +\psi k}(y_{1},...,y_{r}\,)=\ ^{\left( k\right) }E\ ^{\left(
r\right) }\left( -{\small (\mu +\psi k),\beta }_{k+1}{\small ,...,\beta }_{r}%
{\small ;\gamma }_{1}{\small ,...,\gamma }_{r}{\small ;x}_{1}{\small ,...,x}%
_{r}\right)
\end{equation*}
in Theorem \ref{thm:6}, \ref{thm:7} and \ref{thm:8} we have bilinear
generating function relations for the multivariable hypergeometric functions.

On the other hand choosing 
\begin{equation*}
s=r,\text{\ \ \ }\Omega _{\mu +\psi k}(y_{1},y_{2})=\ ^{\left( k\right)
}H_{4}^{\left( r\right) }\left( -\left( \mu +\psi k\right) {\small ,\beta }%
_{k+1}{\small ,...,\beta }_{r}{\small ;\gamma }_{1}{\small ,...,\gamma }_{r}%
{\small ;x}_{1}{\small ,...,x}_{r}\right)
\end{equation*}
in Theorem ~\ref{thm:7}, where the multivariable fourth kind Horn functions 
\cite{E}, generated by 
\begin{align}
\sum \limits_{n=0}^{\infty }&\frac{(\lambda )_{n}}{n!}~^{\left( k\right)
}H_{4}^{\left( r\right) }\left( -{\small n,\beta }_{k+1}{\small ,...,\beta }%
_{r}{\small ;\gamma }_{1}{\small ,...,\gamma }_{r}{\small ;x}_{1}{\small %
,...,x}_{r}\right) t^{n}  \label{d3} \\
& =(1-t)^{-\lambda }~^{\left( k\right) }H\ _{4}^{\left( r\right) }\left( 
{\small \lambda ,\beta }_{k+1}{\small ,...,\beta }_{r}{\small ;\gamma }_{1}%
{\small ,...,\gamma }_{r}{\small ;\frac{x_{1}t^{2}}{\left( 1-t\right) ^{2}}%
,...,\frac{x_{k}t^{2}}{\left( 1-t\right) ^{2}},\frac{-x_{k+1}t}{%
\left(1-t\right) },...,}\frac{-x_{r}t}{\left( 1-t\right) }\right) ,  \notag
\end{align}
where $\left \vert t\right \vert <1$.

We are thus led to the following result which provides a class of bilateral
generating functions for the multivariable hypergeometric functions $%
^{\left( k\right) }E^{\left( r\right) }$\ and the multivariable fourth kind
Horn functions.

\begin{corollary}
\label{cor:6} If 
\begin{align*}
\Lambda _{\mu ,\psi }(y_{1},...,y_{r};\zeta ) &:=\sum
\limits_{k=0}^{\infty}a_{k}\ \ ^{\left( k\right) }H\ _{4}^{\left( r\right)
}\left( -\left( \mu+\psi k\right) {\small ,\beta }_{k+1}{\small ,...,\beta }%
_{r}{\small ;\gamma }_{1}{\small ,...,\gamma }_{r}{\small ;y}_{1}{\small %
,...,y}_{r}\right)\zeta ^{k}~ \\
(a_{k} &\neq 0,\, \, \ \mu ,\psi \in \mathbb{C})
\end{align*}
then, we have 
\begin{align}
\sum \limits_{n=0}^{\infty }&\sum \limits_{k=0}^{[n/p]}a_{k}\
(\lambda_{1})_{n-pk}\ \ ^{\left( k\right) }E\ ^{\left( r\right) }\left( 
{\small -n-pk,\beta }_{k+1}{\small ,...,\beta }_{r}{\small ;\gamma }_{1}%
{\small ,...,\gamma }_{r}{\small ;x}_{1}{\small ,...,x}_{r}\right)  \notag \\
&\times \ ^{\left( k\right) }H\ _{4}^{\left( r\right) }\left( -\left(
\mu+\psi k\right) {\small ,\beta }_{k+1}{\small ,...,\beta }_{r}{\small %
;\gamma}_{1}{\small ,...,\gamma }_{r}{\small ;y}_{1}{\small ,...,y}%
_{r}\right) \frac{\eta ^{k}t^{n-pk}}{\left( n-pk\right) !\ }  \notag \\
&=(1-t)^{-\lambda _{1}}\ ^{\left( k\right) }E\ ^{\left( r\right) }\left( 
{\small \lambda }_{1}{\small ,\beta }_{k+1}{\small ,...,\beta }_{r}{\small %
;\gamma }_{1}{\small ,...,\gamma }_{r}{\small ;\frac{x_{1}\left( -t\right)
^{\rho }}{\left(1-t\right) ^{\rho }},...,\frac{x_{k}\left( -t\right) ^{\rho }%
}{\left(1-t\right) ^{\rho }},\frac{-x_{k+1}t}{\left( 1-t\right) },...,}\frac{%
-x_{r}t}{\left( 1-t\right) }\right)  \notag \\
&\times \text{ }\Lambda _{\mu ,\psi }(y;\eta )  \label{d4}
\end{align}
provided that each member of (\ref{d4}) exists.
\end{corollary}

\begin{remark}
\label{rem:14} Using the generating relation (\ref{d3}) for multivariable
fourth kind Horn functions and getting $a_{k}=\frac{\left( \lambda
_{2}\right) _{k}}{k!},$\ $\mu =0,$\ $\psi =1,$\ in Corollary ~\ref{cor:6},
we find that 
\begin{align*}
\sum \limits_{n=0}^{\infty }&\sum \limits_{k=0}^{[n/p]}\ \left(
\lambda_{2}\right) _{k}\ (\lambda _{1})_{n-pk}\ \ ^{\left( k\right) }E\
^{\left(r\right) }\left( {\small -n-pk,\beta }_{k+1}{\small ,...,\beta }_{r}%
{\small ;\gamma }_{1}{\small ,...,\gamma }_{r}{\small ;x}_{1}{\small ,...,x}%
_{r}\right) \\
&\times \ ^{\left( k\right) }H\ _{4}^{\left( r\right) }\left( -k{\small %
,\beta }_{k+1}{\small ,...,\beta }_{r}{\small ;\gamma }_{1}{\small %
,...,\gamma }_{r}{\small ;y}_{1}{\small ,...,y}_{r}\right) \frac{\eta
^{k}t^{n-pk}}{k!\left( n-pk\right) !} \\
&=\left( 1-t\right) ^{-\lambda _{1}}\ ^{\left( k\right) }E\ ^{\left(
r\right)}\left( {\small \lambda }_{1}{\small ,\beta }_{k+1}{\small %
,...,\beta }_{r}{\small ;\gamma }_{1}{\small ,...,\gamma }_{r}{\small ;\frac{%
x_{1}\left(-t\right) ^{\rho }}{\left( 1-t\right) ^{\rho }},...,\frac{%
,x_{k}\left(-t\right) ^{\rho }}{\left( 1-t\right) ^{\rho }},\frac{-x_{k+1}t}{%
\left(1-t\right) },...,}\frac{-x_{r}t}{\left( 1-t\right) }\right) \\
&\times (1-\eta )^{-\lambda _{2}}~^{\left( k\right) }H\ _{4}^{\left(
r\right)}\left( \lambda _{2}{\small ,\beta }_{k+1}{\small ,...,\beta }_{r}%
{\small ;\gamma }_{1}{\small ,...,\gamma }_{r}{\small ;\frac{y_{1}\eta ^{2}}{%
\left(1-\eta \right) ^{2}},...,\frac{y_{k}\eta ^{2}}{\left( 1-\eta \right)
^{2}},\frac{-y_{k+1}\eta }{\left( 1-\eta \right) },...,}\frac{-y_{r}\eta }{%
\left(1-\eta \right) }\right).
\end{align*}
\end{remark}

Furthermore, for every suitable choice of the coefficients $a_{k}\, \,(k\in
N_{0}),$\ if the multivariable functions $\Omega _{\mu +\psi
k}(y_{1},...,y_{r})$, $r\in N$, are expressed as an appropriate product of
several simpler functions, the assertions of \ Theorems \ref{thm:6}, \ref%
{thm:7}, \ref{thm:8} and \ref{thm:9} can be applied in order to derive
various families of multilinear and multilateral generating functions for
the multivariable hypergeometric functions $^{\left( k\right) }E^{\left(
r\right) }$\ defined by (\ref{b1}).


\bibliographystyle{plain}
\bibliography{mmnsample}

\end{document}